\documentclass[10pt,a4paper]{article}
\usepackage{amsmath,amscd,amsxtra,amsthm,amssymb,makeidx,graphics,booktabs,mathrsfs,titlesec,url}
\theoremstyle{plain}

\newtheorem{theorem}{Theorem}[section]
\newtheorem{lemma}[theorem]{Lemma}
\newtheorem{corollary}[theorem]{Corollary}
\newtheorem{proposition}[theorem]{Proposition}

\newtheorem{example}[theorem]{Example}

\newtheoremstyle{citing}{3pt}{3pt}{\itshape}{0pt}{\bfseries}%
{.}{ }{\thmnote{#3}}\theoremstyle{citing}

\newcommand{\PSL}{\operatorname{PSL}}
\newcommand{\PGL}{\operatorname{PGL}}
\newcommand{\SL}{\operatorname{SL}}
\newcommand{\Gal}{\operatorname{Gal}}
\newcommand{\GL}{\operatorname{GL}}

\begin{document}
\title{Families Of Elliptic Curves With The Same Mod 8 Representations}
\author{Zexiang Chen}
\date{}
\maketitle
\begin{abstract} Let $E:y^2=x^3+ax+b$ be an elliptic curve defined over $\mathbb{Q}$. We compute certain twists of the classical modular curve $X(8)$. Searching for rational points on these twists enables us to find non-trivial pairs of $8$-congruent elliptic curves over $\mathbb{Q}$, i.e. pairs of non-isogenous elliptic curves over $\mathbb{Q}$ whose $8$-torsion subgroups are isomorphic
as Galois modules. We also show that there are infinitely many examples over $\mathbb{Q}$.
\end{abstract}
\section{Introduction And Notation}
$\quad$Let $E,F$ be elliptic curves over $\mathbb{Q}$. For each $n \ge 1$, we say $E$ and $F$ are $n$-congruent with power $r$ if $E[n] \cong
F[n]$ as Galois modules and the Weil pairing is switched to the power of $r \in (\mathbb{Z}/n\mathbb{Z})^*$. That is, if $\phi:E[n] \to F[n]$ is a $G_\mathbb{Q}$-equivariant
isomorphism then $e^r_n(P,Q)=e_n(\phi P,\phi Q)$ for all $P,Q \in E[n]$.

For each elliptic curve $E$, the families of elliptic curves which are $n$-congruent to $E$ with power $r$ are parameterised by a modular curve
$X^r_E(n)$. That is, each non-cuspidal point on $X^r_E(n)$ corresponds to an isomorphism class $(F,\phi)$ where $F$ is an elliptic curve and $\phi:E[n] \to F[n]$ is a $G_\mathbb{Q}$-isomorphism which switches the Weil pairing to the power of $r$.
In fact we only need to focus on $r \in (\mathbb{Z}/n\mathbb{Z})^*$ mod squares. The families of elliptic curves parameterised by $X^1_E(n),n \le 5$
were computed by Rubin and Silverberg [RS] and
the existence and theoretical construction of $X^r_E(n)$ can be found in [S].

When $n \le 5$, $X^r_E(n)$ have genus $0$ and they have infinitely many rational points. The families of elliptic curves parameterised by $X^r_E(n)$ with $r \neq 1$ can be found in [F1], [F2]. When $n \ge 7$,
$X^r_E(n)$ have genus greater than $1$ and so there are only finitely many rational points on each of these.

One of the motivations to study the equations for $n \ge 7$ is to answer Mazur's question [M] which concerns whether there are any pairs of non-isogenous $n$-congruent curves. Motivated by Mazur's question, Kani and Schanz [KS] studied the geometry of the surface that parametrise pairs of $n$-congruent of elliptic curves. This prompted them to conjecture that for any $n \le 12$ there are infinitely many pairs of $n$-congruent non-isogenous elliptic curves over $\mathbb{Q}$. It is understood that we are looking for examples with distinct $j$-invariants, since otherwise from any single
example we could construct infinitely many by taking quadratic twists.

The conjecture in the case $n=7$ was proved by Halberstadt and Kraus [HK] where they gave explicit formula for $X^1_E(7)$. The equation of $X^{6}_E(7)$ was computed by Poonen, Schaefer and Stoll [PSS] where they study the equation of $X^{6}_E(7)$ to solve the Diophantine equation $x^2+y^3=z^7$.
The conjecture in the case $n=9,11$ was proved by Fisher [F3] where he gave explicit formula for $X^r_E(9)$ and $X^r_E(11)$ with
$r=\pm 1$.

In this paper we will give
equations for the $X^r_E(8), r=1,3,5,7$ and the families of elliptic curves parameterised by these modular curves. For convention we write
$X_E(8)$ as $X^1_E(8)$. In the last section we will discuss the relation between the equations we obtain and the classification of modular diagonal surfaces
as described in [KS] which then helps us to generate examples of pairs of non-isogenous $8$-congruent elliptic curves.

Another motivation for studying $n$-congruence of elliptic curves is the following. It was observed by Cremona and Mazur [CM] that if elliptic curves $E$ and
$F$ are $n$-congruent then the Mordell-Weil group of $F$ can sometimes be
used to explain elements of the Tate-Shafarevich group of $E$.

We fix our convention for the classical modular curve. Let $X(n)$ be the classical modular curve on which each non-cuspidal
point corresponds to an isomorphism class $(E,\phi)$ where $E$ is an elliptic curve and
$$\phi: \mathbb{Z}/n\mathbb{Z} \times \mu_n \cong E[n]$$
such that
$$e_n(\phi((a_1,\zeta_1),(a_2,\zeta_2)))=\frac{\zeta^{a_1}_2}{\zeta^{a_2}_1}.$$
Equivalently, each non-cuspidal point corresponds to an isomorphism class $(E,P,C)$ where $P$ is a primitive $n$-torsion point on $E$
and $C$ is a cyclic subgroup of $E$ of order $n$ which does not contain any multiple of $P$. Write
$Y(n)=X(n)\backslash\{\text{cusps}\}$.

We write $\bar{\mathbb{Q}}$ as the algebraic closure of $\mathbb{Q}$ and $G_\mathbb{Q}=\Gal(\bar{\mathbb{Q}}/\mathbb{Q})$.
For each $n \ge 1$ and any field $L$, write $K_n(L)$ as the function field of $X(n)/L$. We will always assume $L$ has characteristic not equal to $2$ or $3$. Let $\mu_n$ be the set
of $n$-th roots of unity. We will always write $\PSL_2(\mathbb{Z}/n\mathbb{Z}):=\SL_2(\mathbb{Z}/n\mathbb{Z})/\{\pm I\}$.

We state our main results. The families of elliptic curves parameterised by $X_E(4)$ and $X^3_E(4)$ will be given in the appendix.
\begin{theorem}
Let $E$ be an elliptic curve with equation $y^2=x^3+ax+b$.  Then
$X_E(8) \subset \mathbb{A}^4_{t,a_0,a_1,a_2}(\mathbb{Q})$
has equations $f_1=g_1=h_1=0$ where
\begin{align*}
f_1&=-aa^2_2+2a_0a_2+a^2_1+\frac{2}{9},\\
g_1&=-2aa_1a_2 - ba^2_2 + 2a_0a_1+\frac{2}{3}t,\\
h_1&=-2ba_1a_2 + a^2_0-t^2+\frac{a}{9},
\end{align*}
with forgetful map $X_E(8) \to X_E(4): (t,a_0,a_1,a_2) \mapsto t$.
\end{theorem}
\begin{theorem}
Let $E$ be an elliptic curve with equation $y^2=x^3+ax+b$ and $D=-4a^3-27b^2$. Then
$X^5_E(8) \subset \mathbb{A}^4_{t,a_0,a_1,a_2}(\mathbb{Q})$
has equations $f_5=g_5=h_5=0$ where
\begin{align*}
f_5&=-aa^2_2 + 2a_0a_2 + a^2_1+\frac{2D}{9},\\
g_5&=-2aa_1a_2 - ba^2_2 + 2a_0a_1+\frac{2D}{3}t,\\
h_5&=-2ba_1a_2 + a^2_0+D(-t^2+\frac{a}{9}).
\end{align*}
with forgetful map $X^5_E(8) \to X_E(4): (t,a_0,a_1,a_2) \mapsto t$.
\end{theorem}
\begin{theorem}
Let $E$ be an elliptic curve with equation $y^2=x^3+ax+b$. Then
$X^3_E(8) \subset \mathbb{A}^4_{t,a_0,a_1,a_2}(\mathbb{Q})$
has equations $f_3=g_3=h_3=0$ where
\begin{align*}
f_3&=-\frac{2}{9}a^2 + 6at^2 + 6bt-(-aa^2_2 + 2a_0a_2 + a^2_1),\\
g_3&=\frac{4}{3}a^2t + \frac{1}{3}ab - 9bt^2-(-2aa_1a_2 - ba^2_2 + 2a_0a_1),\\
h_3&=-\frac{4}{9}a^3 + 4a^2t^2 + 4abt - 2b^2-( -2ba_1a_2 + a^2_0).
\end{align*}
with forgetful map $X^3_E(8) \to X^3_E(4): (t,a_0,a_1,a_2) \mapsto t$.
\end{theorem}
\begin{theorem}
Let $E$ be an elliptic curve with equation $y^2=x^3+ax+b$. Then
$X^7_E(8) \subset \mathbb{A}^4_{t,a_0,a_1,a_2}(\mathbb{Q})$
has equations $f_7=g_7=h_7=0$ where
\begin{align*}
f_7&=3t^2+\frac{a}{9}-aa^2_2 + 2a_0a_2 + a^2_1,\\
g_7&=\frac{4}{3}at+\frac{2}{3}b-2aa_1a_2 - ba^2_2 + 2a_0a_1,\\
h_7&=at^2+2bt-\frac{1}{9}a^2-2ba_1a_2 + a^2_0.
\end{align*}
with forgetful map $X^7_E(8) \to X^3_E(4): (t,a_0,a_1,a_2) \mapsto t$.
\end{theorem}
\begin{flushleft}
{\bf Remark}. The families of elliptic curves parameterised by $X^r_E(8)$ can be read off from the families of elliptic curves parameterised by
$X^{\bar{r}}_E(4)$ via the forgetful map $X^r_E(8) \to X^{\bar{r}}_E(4)$ where $\bar{r}=r$ mod $4$.
\end{flushleft}

We will give basic properties of the modular curve $X(8)$ in Section 2. In Section 3 we will describe the function fields of $X^r_E(8)$ over $\mathbb{Q}$
for each $r=1,3,5,7$. Then we will prove Theorem 1.1 and 1.2 in Section 4, based on the observations in Section 3 and the fact there is always a rational
point on $X_E(8)$. The proofs of Theorem 1.3 and 1.4 require some cocycle
calculations, which we will give in Section 5, and we will prove Theorem 1.3
and 1.4 in Section 6.

\begin{flushleft}
{\bf Acknowledgement}
\end{flushleft}
$\quad$I would like to thank Tom Fisher for useful discussion and advice. Most symbolic computations were done by MAGMA [MAG].
\section{The Modular Curve X(8)}
$\quad$In this section we give the basic properties of the curve $X(8)$. We start by introducing the structure of $X(4)$. Fix a primitive $4$th root of unity $i$. It is well-known (see for example [S]) that the modular curve
$X(4)$ can be identified with $\mathbb{P}^1$ by
$$(E_u,P_u,C_u) \mapsto u$$
where
$$E_u: y^2=x^3-27(256u^8+224u^4+1)x-54(-4096u^{12}+8448u^8+528u^4-1)$$
is an elliptic curve, $P_u=(48u^4 - 144u^3 + 72u^2 - 36u + 3,1728u^5 - 1728u^4 + 864u^3- 432u^2 + 108u)$
is a primitive $4$-torsion point and $C_u$ is generated by $Q_u=(48u^4 - 15,i(864u^4 - 54))$. The cusps of $X(4)$ are points satisfying $u(16u^4-1)=0$ and $u=\infty$.

Take the model $E_u$ and let $x_1$ and $x_2$ be the $x$-coordinates of any half point of $P_u$ and $Q_u$ respectively. They satisfy the vanishing of the following polynomials
\begin{align*}
f&=(x_1-48u^4 + 144u^3 - 72u^2 + 36u - 3)^4\\
&~+1296u(2u-1)^4(4u^2+1)(x_1 - 48u^4 - 72u^2 - 3)^2,\\
g&=(x_2-48u^4+15)^4+1296(16u^4-1)(x_2+96u^4+6)^2.\\
\end{align*}
Solving these directly we conclude that the function field of $X(8)/L$ is
$$K_8(L):=L(u, \sqrt{u^2-1/4}, \sqrt{u^2+1/4},\sqrt{-u})$$ for any field $L$ containing $\mu_8$ and so if $\zeta$ is a fixed $8$th root of unity then a model of $X(8)$ in $\mathbb{A}^4_{u,X_1,X_2,X_3}/L$ is given by
$$X^2_1=u^2-\frac{1}{4},X^2_2=u^2+\frac{1}{4},X^2_3=-u.$$
The projective closure of this is a smooth curve of genus $5$ and the families of elliptic curves parameterized by $X(8)$ are
$$E_{u,X_1,X_2,X_3}: y^2=x^3-27(256u^8+224u^4+1)x-54(-4096u^{12}+8448u^8+528u^4-1)$$
together with a $G_\mathbb{Q}$-invariant $8$-torsion point
$P_{u,X_1,X_2,X_3}=(P_x,P_y)$ and a $G_\mathbb{Q}$-invariant cyclic subgroup $\langle
Q_{u,X_1,X_2,X_3}=(Q_x,Q_y) \rangle$
where
\begin{align*}
P_x&=-36(4X^5_3 + 4X^4_3 + 4X^3_3 + 2X^2_3 + X_3)X_2 + 48X^8_3 +144X^7_3+144X^6_3\\
& + 72X^5_3 + 72X^4_3 + 36X^3_3 + 36X^2_3 + 18X_3 + 3,\\
P_y&=108(16X^9_3 + 32X^8_3 + 32X^7_3 + 32X^6_3 + 24X^5_3 + 16X^4_3 +
    8X^3_3 +4X^2_3\\
&+ X_3)X_2-1728X^{11}_3 - 3456X^{10}_3 - 4320X^9_3- 3456X^8_3 - 2592X^7_3\\
&- 1728X^6_3- 1296X^5_3 - 864X^4_3 - 540X^3_3 -216X^2_3 - 54X_3,\\
Q_x&=-72\zeta^2X_1X_2 + (72(\zeta^3 + \zeta)X^4_3 + 18(\zeta^3 +\zeta))X_1+(72(\zeta^3 -\zeta)X^4_3 \\
&- 18(\zeta^3 - \zeta))X_2+ 48X^8_3 - 15,\\
Q_y&=432X_1X_2 + (864(-\zeta^3 + \zeta)X^8_3 + 432(\zeta^3 - \zeta)X3^4+ 162(\zeta^3 -\zeta))X_1 \\
&+ ((-864\zeta^3 - 864\zeta)X^8_3 + 432(-\zeta^3 - \zeta)X^4_3 + 162(\zeta^3 +\zeta))X_2\\
& + 1728\zeta^2X^8_3 - 108\zeta^2.
\end{align*}
The forgetful morphism $X(8) \to X(4)$ is given by $(u,X_1,X_2,X_3) \mapsto u$. In particular, this is only ramified above the cusps with ramification degree $2$.
The function field of $X(8)$ is obtained by adjoining the square roots of three rational functions of degree $2$ on $X(4)$ and the zeroes of these rational functions are the cusps of $X(4)$.

Let $G_n:=\PSL_2(\mathbb{Z}/n\mathbb{Z})$. It is well-known that $\Gal(K_n(\mathbb{C})/K_1(\mathbb{C})) \cong G_n$. Let $H=\Gal(K_8(\mathbb{C})/K_4(\mathbb{C}))$ then we have an exact sequence
$$1 \to H \to G_8 \to G/H \cong G_4 \to 1$$
and so $H \cong (\mathbb{Z}/2\mathbb{Z})^3$. The group $G_n$ acts on
$X(n)$ by relabeling the $n$-torsion points. Explicitly, for each $\alpha \in G_n$ and any point $(E,\phi) \in Y(n)$, $\alpha$ acts on $(E,\phi)$ by
$$\alpha \circ (E,\phi)=(E,\alpha\circ \phi).$$

\section{The Modular Elliptic Curves}
$\quad$
We firstly recall some results of the level four structure and introduce the algorithm to compute $X_E(4)$ in [F1].
Let $E:y^2=x^3+ax+b$ be an elliptic curve and write $c_4=-\frac{a}{27},c_6=-\frac{b}{54}$.
Take homogenous coordinate $(u:v)$ for $X(4)$ and define
$$c_4(u,v)=256u^8 + 224u^4v^4 + v^8,c_6(u,v)=-4096u^{12} + 8448u^8v^4 + 528u^4v^8 - v^{12}$$

Let $T=uv(16u^4-v^4)$ and $T_u,T_v$ be the partial derivative of $T$
with respect to $u,v$ respectively. Now pick $u,v \in \mathbb{C}$ such that $c_4(u,v)=c_4,c_6(u,v)=c_6$.  Then as is shown in [F1], Lemma 8.4 and Theorem 13.2, the isomorphism $X_E(4) \to X(4)$ is given by
fractional linear map represented by the matrix
$$\begin{pmatrix} u&-T_v\\v&T_u \end{pmatrix}$$
and so the isomorphism $X(4) \to X_E(4)$ is given by fractional linear map
represented by the matrix
$$\begin{pmatrix} T_u&T_v\\-v&u \end{pmatrix}.$$
Under this isomorphism the point $\infty$ on $X_E(4)$ corresponds to $E$ itself. From now on we will identify $X_E(4)$ with $\mathbb{P}^1$ by this isomorphism.

Further, based on the observation in [F1] page 31, we conclude that the curve $X^3_E(4)$ can be chosen to be the same as $X_E(4)$ (with the same isomorphism to $X(4)$) in the sense that if we pick affine coordinate $\mathbb{A}^1_t$ for $X_E(4)$ and
$$E_t: y^2=x^3-27a_E(t)x-54b_E(t)$$ are families of elliptic curves parameterised by $X_E(4)$ then the families of elliptic curves parameterised by $X^3_E(4)$ are
$$E'_t:=E^{\Delta_E}_t: y^2=x^3-27\Delta^2_Ea_E(t)x-54\Delta^3_Eb_E(t).$$ From now on we will write these
to be the families of elliptic curves parameterised by $X^3_E(4)$ and we will
give the expressions of $a_E(t)$ and $b_E(t)$ in the appendix. In fact this identification can be explained by the following
lemma.
\begin{lemma} Let $E$ be an elliptic curve and $E^{\Delta_E}$ be the quadratic
twist of $E$ by its discriminant $\Delta_E$. Let $\gamma:E \to E^{\Delta_E}$ be the natural isomorphism
$$(x,y) \mapsto (x\Delta_E,y\Delta^{\frac{3}{2}}_E).$$
Let $p',q'$ be the image of $p,q$ respectively. Then the map $\phi:E[4] \to E^{\Delta_E}[4]$
$$\phi(p)= p'+2q', \phi(q)=2p'+3q'$$
is a $G_\mathbb{Q}$-equivariant isomorphism.
\end{lemma}
This result can also be found in [BD] Section 7.
\begin{proof} Fix a basis $\{p,q\}$ for $E[4]$. For each $s \in G_\mathbb{Q}$, we identify $s$ with its image under
$\theta': G_\mathbb{Q} \to \GL(E[4]) \subset \GL_2(\mathbb{Z}/4\mathbb{Z})$. Take generators
$v_1,v_2,v_3$ for $\GL_2(\mathbb{Z}/4\mathbb{Z})$ where
$$v_1=\begin{pmatrix} 3&0\\0&1 \end{pmatrix},v_2=\begin{pmatrix}
0&1\\3&0\end{pmatrix}, v_3=\begin{pmatrix} 1&1\\0&1 \end{pmatrix}.$$
Then it suffices to check that $v_j\phi=\phi v_j,j=1,2,3$. Note that $v_1$ fixes $\sqrt{\Delta_E}$ and $v_2,v_3$ switch the sign of $\sqrt{\Delta_E}$.

Then a direct computation shows that
\begin{align*}
v_1\phi(p)&=\phi v_1(p)=3p'+2q', v_1\phi(q)=\phi v_1(q)=2p'+3q',\\
v_2\phi(p)&=\phi v_2(p)=2p'+q', v_2\phi(q)=\phi v_2(q)=p'+2q',\\
v_3\phi(p)&=\phi v_3(p)=p'+2q',v_3\phi(q)=\phi v_3(q)=3p'+q'.
\end{align*}
\end{proof}

\begin{lemma}
Let $t_1,\ldots,t_6$ be the cusps of $X_E(4)$ which are the images of $\pm\frac{1}{2},\pm\frac{i}{2},0,\infty$ respectively, under the isomorphism
$X(4) \to X_E(4)$. If we set
$$m_1=t_1+t_2,m_2=t_3+t_4,m_3=t_5+t_6,l_1=t_1t_2,l_2=t_3t_4,l_3=t_5t_6$$
and let $\theta_j,j=1,2,3$ be the roots of $x^3+ax+b=0$.
Then $m_j=-\frac{2}{3}\theta_j$ and $l_j=-\frac{1}{9}(2\theta^2_j+a)$
for each $j$.
\end{lemma}
\begin{proof} This follows from a direct computation.
\end{proof}
\begin{flushleft}
{\bf Remark}.
Since we identify $X^3_E(4)$ with $X_E(4)$, so $t_1,\ldots,t_6$ are also the cusps of $X^3_E(4)$ and so the above lemma also holds for $X^3_E(4)$.
\end{flushleft}

We now illustrate the method to compute $X^r_E(8)$, $r=1,3,5,7$. For simplicity, assume that $x^3+ax+b$ is irreducible. It follows immediately from compatibility of the Weil pairing that $X^r_E(8)$ is a cover of $X^{\bar{r}}_E(4)$, where $\bar{r}=r$ mod $4$.
It can be shown that $X^r_E(n)$ is a twist of $X(n)$ (see, for example, [S]).
So $X(8)$ and $X^r_E(8)$ have the same ramification behavior under the forgetful morphism to the level four structure. Thus the forgetful morphism
$X^r_E(8) \to X^{\bar{r}}_E(4)$ is only ramified at the points above the cusps of $X^{\bar{r}}_E(4)$.

\begin{lemma}
For each $r \in (\mathbb{Z}/8\mathbb{Z})^*$, the function field of $X^r_E(8)$ over $\mathbb{Q}(E[2])$
is
$$\mathbb{Q}(E[2])\left(t,\sqrt{\alpha_{r,1}(t-t_1)(t-t_2)},\sqrt{\alpha_{r,2}(t-t_3)(t-t_4)},\sqrt{\alpha_{r,3}(t-t_5)(t-t_6)}\right)$$
for some appropriate $\alpha_{r,j} \in \mathbb{Q}(E[2]), j=1,2,3$. We call
these $\alpha_{r,j},j=1,2,3$ the {\bf scaling factors} of $X^r_E(8)$.
\end{lemma}
\begin{proof}
As is described in Section 2, if we fix an affine coordinate $u$ of $X(4)$, then the function field of $X(8)$ over $\mathbb{Q}(\zeta)$ is given by
$$\mathbb{Q}(\zeta)(u,\sqrt{u^2-1/4},\sqrt{u^2+1/4},\sqrt{-u})$$
where $\zeta$ is a fixed primitive $8$th root of unity.

Fix an affine coordinate $t$ of $X_E(4)$ as above. Since $t_1,\ldots,t_6$ are the images of $\pm \frac{1}{2},\pm\frac{i}{2},0,\infty$ respectively, the function field of $X^r_E(8)$ over $\mathbb{C}$ has the form
$$\mathbb{C}(t,\sqrt{(t-t_1)(t-t_2)},\sqrt{(t-t_3)(t-t_4)},\sqrt{(t-t_5)(t-t_6)}).$$
By Lemma 3.2, the rational functions
$(t-t_1)(t-t_2),(t-t_3)(t-t_4),(t-t_5)(t-t_6)$ are defined over $\mathbb{Q}(E[2])$ and are conjugate to each other.

As $X^r_E(8)$ has a model over $\mathbb{Q}$ and so it has a model over $\mathbb{Q}(E[2])$. Then the function field of $X^r_E(8)$ over
$\mathbb{Q}(E[2])$ is
$$\mathbb{Q}(E[2])\left(t,\sqrt{\alpha_{r,1}(t-t_1)(t-t_2)},\sqrt{\alpha_{r,2}(t-t_3)(t-t_4)},\sqrt{\alpha_{r,3}(t-t_5)(t-t_6)}\right)$$
for some appropriate $\alpha_{r,1},\alpha_{r,2},\alpha_{r,3} \in \mathbb{Q}(E[2])$ which are conjugate to each other.
\end{proof}
\begin{corollary}
For each $r \in (\mathbb{Z}/8\mathbb{Z})^*$, the equation of $X^r_E(8) \subset A^4_{t,a_0,a_1,a_2}(\mathbb{Q})$ is determined by the scaling factors $\alpha_{r,j},j=1,2,3$. In particular, the equation of $X^r_E(8)$ over
$\mathbb{Q}$ is obtained
by comparing the coefficients of $1,\theta_j,\theta^2_j,j=1,2,3$ in the equations
$$\alpha_{r,j} (t-t_{2j-1})(t-t_{2j})=(a_0+a_1\theta_j+a_2\theta^2_j)^2, j=1,2,3.$$
\end{corollary}
\begin{proof}
The extension of function fields $(X^r_E(8)/\mathbb{Q}(E[2]))/(X^r_E(8)/\mathbb{Q})$ is Galois. Therefore to find a model of $X^r_E(8)$ over $\mathbb{Q}$, it suffices to find enough generating elements in the function field of $X^r_E(8)$ over $\mathbb{Q}(E[2])$ which are fixed by $\Gal(\mathbb{Q}(E[2])/\mathbb{Q})$. Explicitly,
we will write $w_j:=\sqrt{\alpha_{r,j}(t-t_{2j-1})(t-t_{2j})}$  and so
$w^2_j=\alpha_{r,j}(t-t_{2j-1})(t-t_{2j})$ for each $j=1,2,3$.

By Lemma 3.2, $w_j=a_0+a_1\theta_j+a_2\theta^2_j$ for some
$a_0,a_1,a_2 \in \mathbb{Q}$ for each $j=1,2,3$. Therefore we obtain equations
$$\alpha_{r,j} (t-t_{2j-1})(t-t_{2j})=(a_0+a_1\theta_j+a_2\theta^2_j)^2, j=1,2,3.$$

To find a model of $X^r_E(8)$ over $\mathbb{Q}$, it suffices to compare the coefficients of $1,\theta_j,\theta^2_j$,
$j=1,2,3$ on both sides of the equations above because these are invariant under the action of $\Gal(\mathbb{Q}(E[2])/\mathbb{Q}$).
\end{proof}
\begin{flushleft}
{\bf Remark}.
In fact it suffices to compare the coefficients of $1,\theta_j,\theta^2_j$ in one of the equations
$$\alpha_{r,j} (t-t_{2j-1})(t-t_{2j})=(a_0+a_1\theta_j+a_2\theta^2_j)^2,j=1,2,3$$
because they are conjugate to each other.
\end{flushleft}

\begin{flushleft}
{\bf Remark}. We are free to multiply $\alpha_{r,j}$ by a non-zero squared factor of the form $(u_0+u_1\theta_j+u_2\theta^2_j)^2$
because this leads to a change of coordinate in $a_0,a_1,a_2$.
\end{flushleft}

We can extend the above results to the case when $x^3+ax+b$ is reducible.
For example, if $x^3+ax+b$ splits completely over $\mathbb{Q}$, then
the rational functions $(t-t_{2j-1})(t-t_{2j}),j=1,2,3$ are defined over $\mathbb{Q}$. So
$X^r_E(8) \subset A^4_{t,w_1,w_2,w_3}(\mathbb{Q})$ has equations
$$w^2_j=\alpha_{r,j}(t-t_{2j-1})(t-t_{2j}), j=1,2,3$$
for some appropriate $\alpha_{r,j} \in \mathbb{Q},j=1,2,3$.
This is isomorphic to the ones stated in Theorem 1.1-1.4 because
there is a bijection between $\{a_0,a_1,a_2\}$ and $\{a_0+a_1\theta_j
+a_2\theta^2_j:j=1,2,3\}$ over $\mathbb{Q}$. The case when
$x^3+ax+b$ has exactly one rational root is similar.

\section{The Modular Curves $X_E(8)$ And $X^5_E(8)$}
$\quad$By Corollary 3.4, to find equations of $X^r_E(8)$ over $\mathbb{Q}$, it suffices
to compute the scaling factors $\alpha_{r,j},j=1,2,3$ as introduced in
 Lemma 3.3. We prove
Theorem 1.1 and 1.2 in this section.
\begin{theorem}
We can pick $\alpha_{1,j}$ to be $1$ for each $j=1,2,3$. In particular, we obtain the equation of $X_E(8)$ as stated in Theorem 1.1, together with the forgetful map $X_E(8) \to X_E(4)$ given by
$(t,a_0,a_1,a_2) \mapsto t$.
\end{theorem}
\begin{proof}
There is always a tautological rational point on the curve $X_E(n)$ for any $n$ which corresponds to the pair $(E,[1])$.
The point on $X_E(4)$ corresponding to $(E,[1])$ is given by the point of infinity under the isomorphism we described in Section 3. Since we construct $X_E(8)$ as a cover of $X_E(4)$, there is a point on $X_E(8)$ above $t=\infty$ which corresponds to $(E,[1])$. By a change of coordinate of $a_0,a_1,a_2$, we may take this point to be $t=\infty,a_0=1,a_1=0,a_2=0$.

By corollary 3.4, the equation of $X_E(8)$ over $\mathbb{Q}$ is determined by comparing the coefficients of $1,\theta_j,\theta^2_j$ in the equations
$$\alpha_{1,j} (t-t_{2j-1})(t-t_{2j})=(a_0+a_1\theta_j+a_2\theta^2_j)^2, j=1,2,3$$
Taking homogenous coordinates in the above equations we have
$$\alpha_{1,j} (t-t_{2j-1}s)(t-t_{2j}s)=(a_0+a_1\theta_j+a_2\theta^2_j)^2, j=1,2,3$$
and so the point $t=\infty,a_0=1,a_1=0,a_2=0$ is now $(t:a_0:a_1:a_2:s)=(1:1:0:0:0)$.
Substituting this point into the equations, we conclude that we can
take $\alpha_{1,j},j=1,2,3$ to be $1$.
\end{proof}
By compatibility of the Weil pairing, $X^5_E(8)$ is also a cover of $X_E(4)$. The proof of Theorem 1.2 is based on the following observations.
\begin{lemma} Let $E$ be an elliptic curve and fix any basis $\{P,Q\}$ for
$E[8]$. Then the map
$$\phi:E[8] \to E[8], \phi(P)=5P,\phi(Q)=Q$$
is $G_R$-equivariant where $R=\mathbb{Q}(E[2])$.
\end{lemma}
\begin{proof}
The non-trivial $2$-torsion points $4P,4Q,4P+4Q$ are $R$-rational. Let $s \in G_R$ and write
$$s(P)=A_1P+A_2Q,g(Q)=A_3P+A_4Q.$$
Then $s(4P)=4P$ and $s(4Q)=4Q$. So $A_2,A_3$ are both even. Thus,
$$\phi(s(P))=\phi(A_1P+A_2Q)=5A_1P+A_2Q=5A_1P+5A_2Q=s(\phi(P))$$
and
$$\phi(s(Q))=\phi(A_3P+A_4Q)=5A_3P+A_4Q=A_3P+A_4Q=s(\phi(Q)).$$
\end{proof}
\begin{lemma}
If the modular curves $X^5_E(8)$ and $X_E(8)$ are isomorphic over $K$ as covers of $X_E(4)$, then $\Delta_E$ is a square in $K$.
\end{lemma}
\begin{proof}
Suppose $X^5_E(8) \cong X_E(8)$ as covers of $X_E(4)$ then there exists a $G_K$-equivariant isomorphism
$\phi: E[8] \to E[8]$ such that $\det \phi=5$ and $\phi$ acts trivially on $E[4]$. If we fix a basis $\{P,Q\}$ for $E[8]$ then we can view $\phi$ as
a $2 \times 2$ matrix in terms of its action on $\{P,Q\}$. We only need to consider $\phi$ in $\PGL_2(\mathbb{Z}/8\mathbb{Z})$ because multiplications by $3,5,7$
are automorphisms on $E[8]$ which preserve the Weil pairing. So we have the following possible matrices to consider
$$T_1=\begin{pmatrix} 1&0\\4&5 \end{pmatrix}, T_2=\begin{pmatrix} 1&4\\0&5
\end{pmatrix}, T_3=\begin{pmatrix} 1&0\\0&5 \end{pmatrix},
T_4=\begin{pmatrix} 1&4\\4&5 \end{pmatrix}.$$

We see $T_1=\begin{pmatrix} 1&0\\1&1 \end{pmatrix}^{-1} T_3
\begin{pmatrix} 1&0\\1&1 \end{pmatrix}$,
$T_2=\begin{pmatrix} 1&1\\0&1\end{pmatrix}^{-1} T_3 \begin{pmatrix}
1&1\\0&1\end{pmatrix}$ and so it suffices to consider $T_3$ and $T_4$.

Let $s \in G_K$ and suppose the action of $s$ on $E[8]$ is given by $s(P)=A_1P+A_2Q,s(Q)=A_3P+A_4Q$. If $\phi$ is given by the matrix $T_3$,
then using $s\phi=\phi s$ we conclude $A_2,A_3$ are even. So $A_1,A_4$ are odd because the action of $s$ is invertible.
This implies $s$ fixes $E[2]$ and so $E[2]$ is $G_K$-invariant. In particular $\Delta_E$ is a square
in $K$.

If $\phi$ is given by $T_4$, then a direct computation using $s\phi=\phi s$ shows that $A_2$ and $A_3$ have the same parity
and $A_1+A_2 \equiv A_1+A_3 \equiv A_4$ mod $2$. Suppose $A_2$ and $A_3$ are both even then we have exactly the same situation as
above and so $\Delta_E$ is a square in $K$. Assume $A_2$ and $A_3$ are both odd. If $A_1$ is odd then $A_4$ is even
and we have $s(4P)=4P+4Q,s(4Q)=4P$. So $s(4P+4Q)=4Q$, in which case $\Delta_E$ is a square. If $A_1$ is even then $A_4$ is odd
and we have $s(4P)=4Q,s(4Q)=4P+4Q$. So $s(4P+4Q)=4P$, in which case $\Delta_E$ is again a square.

\end{proof}
\begin{theorem} We can pick $\alpha_{5,j}$ to be $D=-4a^3-27b^2$ for each $j=1,2,3$. In particular, we obtain the equation of $X^5_E(8)$ as stated in Theorem 1.2, together with the forgetful map $X^5_E(8) \to X_E(4)$ given by
$(t,a_0,a_1,a_2) \mapsto t$.
\end{theorem}
\begin{proof}
By the Lemma 4.2, there is a
$\mathbb{Q}(E[2])$-rational point on $X^5_E(8)$ above $t=\infty$ which corresponds to
$(E,\phi)$ where $\phi$ is the same map as in Lemma 4.2.
Therefore $\alpha_{5,j},j=1,2,3$
are squares in $\mathbb{Q}(E[2])$. But there is a unique quadratic subfield
inside $\mathbb{Q}(E[2])$ which is $\mathbb{Q}(\sqrt{D})$ where
$D=-4a^3-27b^2$.

By the last remark of Section 3, $\alpha_{5,j}$ can be multiplied by any non-zero squared factor of the form $(u_0+u_1\theta_j+u_2\theta^2_j)^2$.
This shows that we may pick $\alpha_{5,j},j=1,2,3$ to be $1$ or $D$. But by Lemma 4.3, if $\alpha_{5,j}=1,j=1,2,3$ then $D$ is a square in $\mathbb{Q}$ and so we should pick $\alpha_{5,j}=D$ for each $j$.  A direct
computation gives the equation of $X^5_E(8)$ as in Theorem 1.2.
\end{proof}

\section{Cocycles}
$\quad$The proofs of Theorem 1.1 and Theorem 1.2 are based on the fact there is always a rational point on the curve $X_E(8)$. However this is not always true for $X^3_E(8)$ or $X^7_E(8)$, for any elliptic curve $E$.
We will prove Theorem 1.3 and 1.4 in the next section. By Corollary 3.4, it suffices to compute $\alpha_{3,j}$ and $\alpha_{7,j},j=1,2,3$.

It is shown in [S] that $X^r_E(n)$ are twists of $X(n)$. In particular,
$X^r_E(8)$ are twists of $X_E(8)$ for each $r \in (\mathbb{Z}/8\mathbb{Z})^*$.
By Theorem 2.2 in [AEC], for each curve $C/\mathbb{Q}$, there is a bijection
between the twists of $C/\mathbb{Q}$ and $H^1(G_\mathbb{Q},\text{Isom}(C))$
where $\text{Isom}(C)$ is the isomorphic group of $C$.
In this section, we will describe the relation between the scaling factors $\alpha_{r,j},j=1,2,3$ introduced in Lemma 3.3 and
the element which corresponds to $X^r_E(8)$ in $H^1(G_\mathbb{Q},\text{Isom}(X_E(8)))$. For simplicity, we again assume
that $x^3+ax+b$ is irreducible.

\begin{lemma} For each $r$, let $\tau$ be an automorphism on $E[8]$ which switches the Weil pairing to the power of $r$. Then for each $s \in G_\mathbb{Q}$, $s\mapsto (^s\tau)\tau^{-1}$ defines a cocycle in
$H^1(G_\mathbb{Q},\text{Isom}(X_E(8)))$ which corresponds to $X^r_E(8)$.
\end{lemma}
\begin{proof}
For each $s \in \mathbb{Q}$, $(^s \tau)\tau^{-1}$ is an automorphism on $E[8]$ preserving
the Weil pairing, which induces an automorphism on $X_E(8)$. Note $[-1]$ acts trivially on $X_E(8)$. Then following a similar argument in [S], we conclude that the curve
corresponding to this cocycle is $X^r_E(8)$.
\end{proof}
\begin{flushleft}
{\bf Remark}. If $(^s \tau)\tau^{-1}$ acts trivially on $E[4]$ modulo $[-1]$ for all $s \in G_\mathbb{Q}$ then
we have an isomorphism between $X_E(8)$ and $X^r_E(8)$ respecting the level four structure.
\end{flushleft}

The group $H \cong (\mathbb{Z}/2\mathbb{Z})^3$ is defined to be the kernel of the reduction map $\PSL_2(\mathbb{Z}/8\mathbb{Z}) \to \PSL_2(\mathbb{Z}/4\mathbb{Z})$ in Section 2 and $H$ is a subgroup of 
Isom$(X_E(8))$.
 Define $H'$ to be the kernel of
$$\GL_2(\mathbb{Z}/8\mathbb{Z})/\{\pm I\} \to
\GL_2(\mathbb{Z}/4\mathbb{Z})/\{\pm I,\pm v\}$$
where
$$v=\begin{pmatrix} 1&2\\2&3 \end{pmatrix}$$
and it can be checked that $H'$ is Abelian. By Lemma 3.1 the matrix $v$ induces a $G_\mathbb{Q}$-equivariant isomorphism between $E[4]$ and $E^{\Delta_E}[4]$ which switches the Weil pairing to the power of $3$, and so $v$ identifies $X^3_E(4)$ with $X_E(4)$.

Since $H$ is a subgroup of $H'$ and $\det v \neq 1$, the following sequence
\begin{center}
$\begin{CD}
0 @> >> H @> >> H' @> \det >> (\mathbb{Z}/8\mathbb{Z})^* @> >> 0\\
\end{CD}$
\end{center}
is exact.
Viewing $g$ as an automorphism on $E[8]$ modulo $[-1]$, we have a Galois action $^s g$ for each $s \in G_\mathbb{Q}$.
Further we have trivial Galois action on $(\mathbb{Z}/8\mathbb{Z})^*$.

Now viewing $H,H',(\mathbb{Z}/8\mathbb{Z})^*$ as $G_\mathbb{Q}$-module we obtain a long exact sequence and in particular we obtain the connecting map
$$(\mathbb{Z}/8\mathbb{Z})^* \to H^1(G_\mathbb{Q},H).$$
The image of $r \in (\mathbb{Z}/8\mathbb{Z})^*$ can be computed as follows.
Pick a lift $v'$ of $r$ in $H'$. Then the image of $r$ in
$H^1(G_\mathbb{Q},H)$ is $s \mapsto (^sv') v'^{-1}$ for each $s \in G_\mathbb{Q}$. Therefore,
$X^r_E(8)$ is the curve corresponding to this cocycle by Lemma 5.1.

Recall that each non-cuspidal point on
$X^r_E(n)$ corresponds to a pair $(F,\phi)$ where $F$ is an elliptic curve and $\phi:E[n] \to F[n]$ is a
$G_\mathbb{Q}$-equivariant isomorphism which switches the Weil pairing to the power of $r$. We consider the image of $7$.
\begin{lemma}
The image of $7$ under $(\mathbb{Z}/8\mathbb{Z})^* \to H^1(G_\mathbb{Q},H)$ induces an isomorphism $\psi: X^7_E(8) \to X_E(8)$ subject to the following commutative diagram
\begin{center}
$\begin{CD}
X^7_E(8) @>\psi>> X_E(8)\\
@VV V @VV  V\\
X^3_E(4) @>\eta>> X_E(4)
\end{CD}$
\end{center}
where $\psi(F,\phi)=(F,\phi\circ v')$ and $\eta(F,\phi)=(F,\phi \circ v)$.
\end{lemma}
\begin{proof} For each $s \in G_\mathbb{Q}$, since $^s\phi=\phi$,
$$(^s\psi)\psi^{-1}(F,\phi)=(F,\phi \circ (^s v')v'^{-1}), (^s \eta)\eta^{-1}(F,\phi)=(F,\phi \circ (^s v)v^{-1}).$$
The Galois conjugate $(^s\psi)\psi^{-1}$ induces an automorphism on $X_E(8)$ which can be read off from
$(^s v')v'^{-1}$. So $\psi$ corresponds to the cocycle $s\mapsto (^s v')v'^{-1}$ which is the image of $7$. The diagram commutes
because $v' \equiv v$ mod $4$.
\end{proof}

We describe the image of $7$ in $H^1(G_\mathbb{Q},H)$ explicitly.
\begin{lemma} Let $v'=\begin{pmatrix} 1&2\\6&3 \end{pmatrix}$ be a lift of $7$ in $H'$. For each $s \in G_\mathbb{Q}$,
we identify $s$ with its image under $\theta: G_\mathbb{Q} \to \GL(E[8]) \subset \GL_2(\mathbb{Z}/8\mathbb{Z})$. Then the action
of $s$ on $v'$ is given by conjugation. Take generators $s_1,s_2,s_3,s_4$ for $\GL_2(\mathbb{Z}/8\mathbb{Z})$ where
$$s_1=\begin{pmatrix} 7&0\\0&1 \end{pmatrix}, s_2=\begin{pmatrix} 5&0\\0&1 \end{pmatrix},
s_3=\begin{pmatrix} 0&1\\-1&0 \end{pmatrix}, s_4=\begin{pmatrix} 1&1\\0&1 \end{pmatrix}.$$
Let $C_{s_j}=(^{s_j}v')v'^{-1}=s_jv's^{-1}_jv'^{-1}$. Then
$$C_{s_1}=\begin{pmatrix} 1&4\\4&1\end{pmatrix},C_{s_2}=\begin{pmatrix}
1&0\\0&1 \end{pmatrix},C_{s_3}=\begin{pmatrix} 3&4\\4&3 \end{pmatrix},
C_{s_4}=\begin{pmatrix} 1&0\\4&1 \end{pmatrix}.$$
\end{lemma}
\begin{proof} This follows from a direct computation.
\end{proof}
Lemma 5.2 and 5.3 give concrete descriptions of $X^7_E(8)$ in terms of the image of $7$ under $(\mathbb{Z}/8\mathbb{Z})^*
\to H^1(G_\mathbb{Q},H^1)$. On the other hand, the equation of $X^7_E(8)$ is determined by the scaling factors
$\alpha_{7,j},j=1,2,3$ by Corollary 3.4. The following lemmas show how these scaling factors are related to the image of $7$ in $H^1(G_\mathbb{Q},H)$.

\begin{lemma} Let $T_1,T_2,T_3$ be the non-trivial $2$-torsion points of $E$
and $M$ be the group $\text{Map}(E[2]\backslash\{O\},\mu_2)$ where the group operation is defined by
$(\chi_1\circ \chi_2)(T_j)=\chi_1(T_j)\chi_2(T_j),j=1,2,3$.
For each $s \in G_\mathbb{Q}$, we define the action $^s \chi$ by $\chi s^{-1}$ as we have trivial action on $\mu_2$.
Then $H \cong M$ as $G_\mathbb{Q}$-module and hence $H^1(G_\mathbb{Q},H) \cong L^*/(L^*)^2$ where
$L=\mathbb{Q}[x]/(x^3+ax+b)$.
\end{lemma}
\begin{proof}
Fix a basis $\{P,Q\}$ for $E[8]$ such that $4P=T_1,4Q=T_2$. Take generators $S_1,S_2,S_3$ for $H$ where
$$S_1=\begin{pmatrix} 1&4\\4&1 \end{pmatrix}, S_2=\begin{pmatrix} 3&4\\4&3 \end{pmatrix}, S_3=\begin{pmatrix} 1&0\\4&1
\end{pmatrix}.$$
For each $s \in G_\mathbb{Q}$,
we identify $s$ with its image under $\theta: G_\mathbb{Q} \to \GL(E[8]) \subset \GL_2(\mathbb{Z}/8\mathbb{Z})$
and the action of $G_\mathbb{Q}$ on $H$ is given by conjugation $^s S_i=sS_is^{-1}$.
We take generators $s_1,s_2,s_3,s_4$ for $\GL_2(\mathbb{Z}/8\mathbb{Z})$ as in Lemma 5.3.

We identify each element $\chi \in M$ with a triple $(e_1,e_2,e_3)$ where $e_i \in \{\pm 1\}$ in the sense that
$\chi(T_i)=e_i$. The action of $G_\mathbb{Q}$ on $M$ is given by $^s \chi=\chi s^{-1}$.

Now define $\pi: H \to M$ explicitly by $S_1 \mapsto \chi_1,S_2 \mapsto \chi_2, S_3 \mapsto \chi_3$ where
$$\chi_1=(-1,-1,1), \chi_2=(1,1,-1), \chi_3=(1,-1,1).$$
Then a direct computation shows that $^{s_i}\pi(S_j)=\pi(^{s_i}S_j)$ for $i=1,2,3,4$ and $j=1,2,3$ and so $\pi$ is a
$G_\mathbb{Q}$-equivariant isomorphism. So $H^1(G_\mathbb{Q},H) \cong H^1(G_\mathbb{Q},M)$. Finally, by Shapiro's lemma and Hilbert 90, $H^1(G_\mathbb{Q},M) \cong L^*/(L^*)^2$.
\end{proof}
Since we assume that $x^3+ax+b$ is irreducible, so $L \cong L_j$ for any $j=1,2,3$ where $L_j=\mathbb{Q}(\theta_j)$,
and we have an embedding $L \hookrightarrow \prod_{j=1}^3 L_j$.
\begin{lemma}
The image of $7$ under $(\mathbb{Z}/8\mathbb{Z})^* \to H^1(G_\mathbb{Q},H) \cong H^1(G_\mathbb{Q},M) \cong
L^*/(L^*)^2 \hookrightarrow \prod_{j=1}^3 L^*_j/(L^*_j)^2$
is  $\left(\alpha_{7,1},\alpha_{7,2},\alpha_{7,3}\right)$.
\end{lemma}
\begin{proof} By considering the function field of $X^7_E(8)$ and $X_E(8)$ over $\mathbb{Q}(E[2])$ (Lemma 3.3),
the map $\psi': \sqrt{\alpha_{7,j}(t-t_{2j-1})(t-t_{2j})} \mapsto \sqrt{\alpha_{1,j}(t-t_{2j-1})(t-t_{2j})}$, $j=1,2,3$ induces an isomorphism $X^7_E(8) \to X_E(8)$ over $\mathbb{Q}(E[2])$. Moreover we have the following commutative diagram
\begin{center}
$\begin{CD}
X^7_E(8) @>\psi'>> X_E(8)\\
@VV V @VV  V\\
X^3_E(4) @>=>> X_E(4)
\end{CD}$
\end{center}
For each $s \in G_{\mathbb{Q}}$, $s$ acts on $E[2]$ by permuting $\{T_1,T_2,T_3\}$. Let $\sigma_{s}$ be
the element in the symmetric group of $\{1,2,3\}$ which corresponds to the action of $s$ on $\{T_1,T_2,T_3\}$.
A direct computation shows that the Galois conjugate $(^s \psi')\psi'^{-1}$ acts on
$X_E(8)$ by
$$\sqrt{\alpha_{1,j}(t-t_{2j-1})(t-t_{2j})} \mapsto \frac{s\left(\sqrt{\frac{\alpha_{1,\sigma^{-1}_s(j)}}{\alpha_{7,\sigma^{-1}_s(j)}}}\right)}{\sqrt{\frac{\alpha_{1,j}}{\alpha_{7,j}}}}
\sqrt{\alpha_{1,i}(t-t_{2j-1})(t-t_{2j})}, j=1,2,3.$$
This induces a cocycle in $H^1(G_\mathbb{Q},M)$,
$$s \mapsto \left( \frac{s\left(\sqrt{\frac{\alpha_{1,\sigma^{-1}_s(1)}}{\alpha_{7,\sigma^{-1}_s(1)}}}\right)}{\sqrt{\frac{\alpha_{1,1}}{\alpha_{7,1}}}},
\frac{s\left(\sqrt{\frac{\alpha_{1,\sigma^{-1}_s(2)}}{\alpha_{7,\sigma^{-1}_s(2)}}}\right)}{\sqrt{\frac{\alpha_{1,2}}{\alpha_{7,2}}}},
\frac{s\left(\sqrt{\frac{\alpha_{1,\sigma^{-1}_s(3)}}{\alpha_{7,\sigma^{-1}_s(3)}}}\right)}{\sqrt{\frac{\alpha_{1,3}}{\alpha_{7,3}}}}
\right).$$
$\psi'$ is an isomorphism from $X^7_E(8)$ to $X_E(8)$ which
fixes the level four structure. So by Lemma 3.1 and Lemma 5.2 this cocycle corresponds to the image of $7$ under the connecting map
$(\mathbb{Z}/7\mathbb{Z})^* \to H^1(G_\mathbb{Q},H)$.
Then by Shapiro's lemma and Hilbert 90, we see that
$\left(\frac{\alpha_{7,1}}{\alpha_{1,1}},\frac{\alpha_{7,2}}{\alpha_{1,2}},
\frac{\alpha_{7,3}}{\alpha_{1,3}}\right),$ is the image of $7$
under
$$(\mathbb{Z}/8\mathbb{Z})^* \to H^1(G_\mathbb{Q},H)
\cong H^1(G_\mathbb{Q},M)
\cong L^*/(L^*)^2 \hookrightarrow \prod_{i=j}^3 L^*_j/(L^*_j)^2.$$
Finally by Theorem 4.1, we can just take $\alpha_{1,j}=1,j=1,2,3$.
\end{proof}
In Section 6 we will take some suitable $\delta_j \in L_j,j=1,2,3$.
To check $\alpha_{7,j}$ can be chosen to be $\delta_j$ for
each $j$, it then suffices to check that the preimage of $\left(\delta_1,\delta_2,\delta_3\right)$
under $H^1(G_\mathbb{Q},H) \cong \prod_{j=1}^3 L^*_j/(L^*_j)^2$
is exactly the same as the image of $7$ under $(\mathbb{Z}/8\mathbb{Z})^* \to H^1(G_\mathbb{Q},H)$, which can be read off from $C_{s_j},j=1,2,3,4$ in Lemma 5.3.
\begin{flushleft}
{\bf Remark}. We can extend the results in the case when $x^3+ax+b$ is reducible. For example, if $x^3+ax+b$ splits completely over
$\mathbb{Q}$, then the Galois action on $M$ is trivial and so we get
$$H^1(G_\mathbb{Q},M) \cong L^*/(L^*)^2 \cong \mathbb{Q}^*/(\mathbb{Q}^*)^2 \times \mathbb{Q}^*/(\mathbb{Q}^*)^2 \times
\mathbb{Q}^*/(\mathbb{Q}^*)^2$$
directly by Hilbert 90. The case when $x^3+ax+b=0$ has exactly
 one rational root is similar.
\end{flushleft}

\section{The Curves $X^3_E(8)$ and $X^7_E(8)$}
$\quad$
We will prove Theorem 1.4 following the strategy we described in Section 5.
We will pick suitable $\delta_j \in \mathbb{Q}(E[2]),j=1,2,3$ which
 are conjugate to each other and show that
$\alpha_{7,j}$ can indeed be chosen to be $\delta_j$ for each $j$. In particular, we will compute the preimage of
$\left(\delta_1,\delta_2,\delta_3\right)$
under
$$H^1(G_\mathbb{Q},H) \cong H^1(G_\mathbb{Q},M) \cong
L^*/(L^*)^2 \hookrightarrow \prod_{j=1}^3 L^*_j/(L^*_j)^2$$
and check it is the same as the image of $7$ under $(\mathbb{Z}/8\mathbb{Z})^*
\to H^1(G_\mathbb{Q},H)$ by using Lemma 5.3, 5.4 and 5.5.

\begin{lemma} Let $E$ be an elliptic curve with equation $y^2=x^3+ax+b$. Let $\theta_j,j=1,2,3$ be the roots of $x^3+ax+b=0$ and
$$\delta_1=(\theta_1-\theta_2)(\theta_3-\theta_1),
\delta_2=(\theta_1-\theta_2)(\theta_2-\theta_3),
\delta_3=(\theta_2-\theta_3)(\theta_3-\theta_1).$$
Then the $x$-coordinates of the primitive $4$-torsion points of $E$ are given by
$$\theta_1\pm i\sqrt{\delta_1},\theta_2\pm i\sqrt{\delta_2},\theta_3\pm i\sqrt{\delta_3}.$$
\end{lemma}
\begin{proof}
This follows immediately from factorising the $4$-division polynomial of $E$
over $\mathbb{Q}(E[2])$.
\end{proof}

We now fix a basis $\{P,Q\}$ for $E[8]$ such that $2P,2Q,2P+2Q$
have $x$-coordinates $\theta_1+i\sqrt{\delta_1}$, $\theta_2+i\sqrt{\delta_2}$, and $\theta_3+i\sqrt{\delta_3}$ respectively by Lemma 6.1. Let $T_1=4P,T_2=4Q,T_3=4P+4Q$ be the non-trivial $2$-torsion points so that $T_1=(\theta_1,0),T_2=(\theta_2,0)$ and
$T_3=(\theta_3,0)$.
\begin{lemma} For each $s \in G_\mathbb{Q}$, we identify $s$ with its image under $\theta: G_\mathbb{Q}
\to \GL(E[8]) \subset \GL_2(\mathbb{Z}/8\mathbb{Z})$. Fix generators $s_1,s_2,s_3,s_4$ for $\GL_2(\mathbb{Z}/8\mathbb{Z})$
as in Lemma 5.3. Then
\begin{align*}
s_1(\sqrt{\delta_1})&=-\sqrt{\delta_1},s_1(\sqrt{\delta_2})=-\sqrt{\delta_2},s_1(\sqrt{\delta_3})=\sqrt{\delta_3},\\
s_2(\sqrt{\delta_1})&=\sqrt{\delta_1},s_2(\sqrt{\delta_2})=\sqrt{\delta_2},s_2(\sqrt{\delta_3})=\sqrt{\delta_3},\\
s_3(\sqrt{\delta_1})&=\sqrt{\delta_2},s_3(\sqrt{\delta_2})=\sqrt{\delta_1},s_3(\sqrt{\delta_3})=-\sqrt{\delta_3}\\
s_4(\sqrt{\delta_1})&=\sqrt{\delta_1},s_4(\sqrt{\delta_2})=\sqrt{\delta_3},s_4(\sqrt{\delta_3})=-\sqrt{\delta_2}.
\end{align*}
\end{lemma}
\begin{proof} Fix a primitive $8$th root of unity $\zeta$ so that $\zeta^2=i$. We have
$$s_1=\begin{pmatrix} 7&0\\0&1 \end{pmatrix}, s_2=\begin{pmatrix} 5&0\\0&1 \end{pmatrix},
s_3=\begin{pmatrix} 0&1\\-1&0 \end{pmatrix}, s_4=\begin{pmatrix} 1&1\\0&1 \end{pmatrix}.$$
Since $s_j(\zeta)=\zeta^{\det s_j}$ so $s_1(\zeta)=\zeta^7,s_2(\zeta)=\zeta^5,s_3(\zeta)=\zeta,
s_4(\zeta)=\zeta$. Therefore $s_1(i)=-i,s_2(i)=i,s_3(i)=i,s_4(i)=i$. The actions of $s_j,j=1,2,3,4$ on $E[4]$ are given by
\begin{align*}
s_1(2P)&=-2P,s_1(2Q)=2Q,s_1(2P+2Q)=-2P+2Q,\\
s_2(2P)&=2P,s_2(2Q)=2Q,s_2(2P+2Q)=2P+2Q,\\
s_3(2P)&=-2Q,s_3(2Q)=2P,s_3(2P+2Q)=2P-2Q,\\
s_4(2P)&=2P,s_4(2Q)=2P+2Q,s_4(2P+2Q)=4P+2Q.
\end{align*}
By considering the $x$-coordinates of these points, we have
\begin{align*}
s_1(\theta_1+i\sqrt{\delta_1})&=\theta_1+i\sqrt{\delta_1},s_1(\theta_2+i\sqrt{\delta_2})=\theta_2+i\sqrt{\delta_2},
s_1(\theta_3+i\sqrt{\delta_3})=\theta_3-i\sqrt{\delta_3},\\
s_2(\theta_1+i\sqrt{\delta_1})&=\theta_1+i\sqrt{\delta_1},s_2(\theta_2+i\sqrt{\delta_2})=\theta_2+i\sqrt{\delta_2},
s_2(\theta_3+i\sqrt{\delta_3})=\theta_3+i\sqrt{\delta_3},\\
s_3(\theta_1+i\sqrt{\delta_1})&=\theta_2+i\sqrt{\delta_2},s_3(\theta_2+i\sqrt{\delta_2})=\theta_1+i\sqrt{\delta_1},
s_3(\theta_3+i\sqrt{\delta_3})=\theta_3-i\sqrt{\delta_3},\\
s_4(\theta_1+i\sqrt{\delta_1})&=\theta_1+i\sqrt{\delta_1},s_4(\theta_2+i\sqrt{\delta_2})=\theta_3+i\sqrt{\delta_3},
s_4(\theta_3+i\sqrt{\delta_3})=\theta_2-i\sqrt{\delta_2}.
\end{align*}
By considering the actions of $s_j,j=1,2,3,4$ on $E[2]$ we have
\begin{align*}
s_1(\theta_1)&=\theta_1,s_1(\theta_2)=\theta_2,s_1(\theta_3)=\theta_3,\\
s_2(\theta_1)&=\theta_1,s_2(\theta_2)=\theta_2,s_2(\theta_3)=\theta_3,\\
s_3(\theta_1)&=\theta_2,s_3(\theta_2)=\theta_1,s_3(\theta_3)=\theta_3,\\
s_4(\theta_1)&=\theta_1,s_4(\theta_2)=\theta_3,s_4(\theta_3)=\theta_2.
\end{align*}
Therefore, we conclude that
\begin{align*}
s_1(\sqrt{\delta_1})&=-\sqrt{\delta_1},s_1(\sqrt{\delta_2})=-\sqrt{\delta_2},s_1(\sqrt{\delta_3})=\sqrt{\delta_3},\\
s_2(\sqrt{\delta_1})&=\sqrt{\delta_1},s_2(\sqrt{\delta_2})=\sqrt{\delta_2},s_2(\sqrt{\delta_3})=\sqrt{\delta_3},\\
s_3(\sqrt{\delta_1})&=\sqrt{\delta_2},s_3(\sqrt{\delta_2})=\sqrt{\delta_1},s_3(\sqrt{\delta_3})=-\sqrt{\delta_3}\\
s_4(\sqrt{\delta_1})&=\sqrt{\delta_1},s_4(\sqrt{\delta_2})=\sqrt{\delta_3},s_4(\sqrt{\delta_3})=-\sqrt{\delta_2}.
\end{align*}
\end{proof}
Each $s_j,j=1,2,3,4$ acts on $E[2]$ by permuting $\{T_1,T_2,T_3\}$. So for each $j$ we write $\sigma_{s_j}$ to be the element in
the symmetric group of $\{1,2,3\}$ which corresponds to the action of $s_j$ on $\{T_1,T_2,T_3\}$.
\begin{lemma}  We have
\begin{align*}
&\frac{s_1\left(\sqrt{\delta_{\sigma^{-1}_{s_1}(1)}}\right)}{\sqrt{\delta_1}}=-1,
\frac{s_1\left(\sqrt{\delta_{\sigma^{-1}_{s_1}(2)}}\right)}{\sqrt{\delta_2}}=-1,
\frac{s_1\left(\sqrt{\delta_{\sigma^{-1}_{s_1}(3)}}\right)}{\sqrt{\delta_3}}=1,\\
&\frac{s_2\left(\sqrt{\delta_{\sigma^{-1}_{s_2}(1)}}\right)}{\sqrt{\delta_1}}=1,
\frac{s_2\left(\sqrt{\delta_{\sigma^{-1}_{s_2}(2)}}\right)}{\sqrt{\delta_2}}=1,
\frac{s_2\left(\sqrt{\delta_{\sigma^{-1}_{s_2}(3)}}\right)}{\sqrt{\delta_3}}=1,\\
&\frac{s_3\left(\sqrt{\delta_{\sigma^{-1}_{s_3}(1)}}\right)}{\sqrt{\delta_1}}=1,
\frac{s_3\left(\sqrt{\delta_{\sigma^{-1}_{s_3}(2)}}\right)}{\sqrt{\delta_2}}=1,
\frac{s_3\left(\sqrt{\delta_{\sigma^{-1}_{s_3}(3)}}\right)}{\sqrt{\delta_3}}=-1,\\
&\frac{s_4\left(\sqrt{\delta_{\sigma^{-1}_{s_4}(1)}}\right)}{\sqrt{\delta_1}}=1,
\frac{s_4\left(\sqrt{\delta_{\sigma^{-1}_{s_4}(2)}}\right)}{\sqrt{\delta_2}}=-1,
\frac{s_4\left(\sqrt{\delta_{\sigma^{-1}_{s_4}(3)}}\right)}{\sqrt{\delta_3}}=1.
\end{align*}
\end{lemma}
\begin{proof} This follows from a direct computation by using Lemma 6.2.
\end{proof}

We now prove Theorem 1.4.
\begin{proof}
We identify each $s \in G_\mathbb{Q}$ with its image under
$\theta: G_\mathbb{Q} \to \GL(E[8]) \subset \GL_2(\mathbb{Z}/8\mathbb{Z})$ and pick generators $s_1,s_2,s_3,s_4$ for
$\GL_2(\mathbb{Z}/8\mathbb{Z})$ as in Lemma 5.3. Then by Lemma 6.3 and
 Shapiro's Lemma, the preimage of $(\delta_1,\delta_2,\delta_3)$
under $H^1(G_\mathbb{Q},M) \cong L^*/(L^*)^2 \hookrightarrow \prod_{j=1}^3 L^*_j/(L^*_j)^2$ is a cocycle $c_s$ which can be described as
$$c_{s_1}=(-1,-1,1),c_{s_2}=(1,1,1),c_{s_3}=(1,1,-1), c_{s_4}=(1,-1,1).$$
By Lemma 5.4, the preimage of $c_{s_j}$ under $H^1(G_\mathbb{Q},H) \cong H^1(G_\mathbb{Q},M)$
is $C_{s_j}$ for each $j=1,2,3,4$, where $C_{s_j},j=1,2,3,4$ are matrices given in Lemma 5.3. But by Lemma 5.3,
$C_{s_j},j=1,2,3,4$ are used to describe the image of $7$ under $(\mathbb{Z}/8\mathbb{Z})^* \to H^1(G_\mathbb{Q},H)$.
This shows that the image of $7$ under
$$(\mathbb{Z}/8\mathbb{Z})^* \to H^1(G_\mathbb{Q},H) \cong H^1(G_\mathbb{Q},M) \cong L^*/(L^*)^2 \hookrightarrow
\prod_{j=1}^3 L^*_j/(L^*_j)^2$$
is $(\delta_1,\delta_2,\delta_3)$.
Then by Lemma 5.5, $\alpha_{7,j}$ can be chosen to be $\delta_j$ for each $j$.
Theorem 1.4 follows from comparing the coefficients
of $1,\theta_j,\theta^2_j$ in the equations
$$\alpha_{7,j}(t-t_{2j-1})(t-t_{2j})=(a_0+a_1\theta_j+a_2\theta^2_j)^2, j=1,2,3.$$
\end{proof}
We now prove Theorem 1.3.
\begin{proof} The connecting map
$(\mathbb{Z}/8\mathbb{Z})^* \to H^1(G_\mathbb{Q},H)$ is a group homomorphism. Therefore, the image of $5$ under
$$(\mathbb{Z}/8\mathbb{Z})^* \to H^1(G_\mathbb{Q},H) \cong H^1(G_\mathbb{Q},M) \cong L^*/(L^*)^2 \hookrightarrow
\prod_{j=1}^3 L^*_j/(L^*_j)^2$$
is the product of the image of $3$ and the image of $7$. So
$\alpha_{3,j}=\alpha_{5,j}\cdot \alpha_{7,j}$ in $L^*_j/(L^*_j)^2$. We have shown in Theorem 4.4 that $\alpha_{5,j}=D$
for each $j=1,2,3$ where $D=-4a^3-27b^2$. Therefore,
$$\alpha_{3,1}=D\alpha_{7,1}=
(\theta_2-\theta_3)^2(\theta_1-\theta_2)^3(\theta_3-\theta_1)^3.$$
Since $\left((\theta_1-\theta_2)(\theta_3-\theta_1)\right)^2$ is a square in $L_1$ so we can take
$\alpha_{3,1}$ to be $(\theta_2-\theta_3)^2(\theta_1-\theta_2)(\theta_3-\theta_1)$. Similarly we can rescale $\alpha_{3,2}$ and $\alpha_{3,3}$ so that
$$\alpha_{3,2}=(\theta_3-\theta_1)^2(\theta_1-\theta_2)(\theta_2-\theta_3)
,\alpha_{3,3}=(\theta_1-\theta_2)^2(\theta_3-\theta_1)(\theta_2-\theta_3).$$
Theorem 1.3 follows from comparing the
coefficients of $1,\theta_j,\theta^2_j$ in the equation
$$\alpha_{3,j}(t-t_{2j-1})(t-t_{2j})=(a_0+a_1\theta_j+a_2\theta^2_j)^2,j=1,2,3.$$
\end{proof}

\begin{flushleft}
{\bf Remark}. The points on $X^r_E(8)$, $r=1,3,5,7$, appear in pairs. In
other words,
if $(t,a_0,a_1,a_2) \in X^r_E(8)$ then $(t,-a_0,-a_1,-a_2) \in X^r_E(8)$
because there is a non-trivial automorphism on $X^r_E(8)$ given by
$$(F,\phi) \mapsto (F,\phi \circ [3]).$$
\end{flushleft}

\begin{flushleft}
{\bf Remark}. Theorem 1.1-1.4 can be generalised to any field of characteristic not equal to $2$ or $3$ by exactly the same method.
\end{flushleft}

\section{The Modular Diagonal Surfaces}
For each $n \ge 1$ and $\epsilon \in (\mathbb{Z}/n\mathbb{Z})^*$, Kani and Schanz classify the type of modular diagonal surface $Z_{n,\epsilon}$ which are constructed as
the quotient of $X(n) \times X(n)$ by
$$\Delta_\epsilon=\{(g,\alpha_\epsilon(g)): g \in \PSL_2(\mathbb{Z}/n\mathbb{Z})\}$$
where $\alpha_\epsilon \in \text{Aut}(\PSL_2(\mathbb{Z}/n\mathbb{Z}))$ is defined by conjugation by the element
$\begin{pmatrix} \epsilon&0\\0&1 \end{pmatrix}$

Each point on the surface
corresponds to a pair of elliptic curves which are $n$-congruent and the Weil pairing is switched to the power of $\epsilon$ [KS].
We are now going to study briefly of these surfaces in terms of the models of $X^r_E(8),r=1,3,5,7$ we got
and explain how it helps to give numerical examples. In particular, we will show that there are infinitely many pairs of non-isogenous elliptic curves which are $8$-congruent with power $r$.
\begin{flushleft}
{\bf Remark}
\end{flushleft}
By a result of Mazur, there are only finitely many $l$ such that rational $l$-isogeny exists and so we only have finitely many sections on the surface which correspond to copies of $X_0(l)$.
To find infinitely many pairs of non-isogenous elliptic curves which are  $8$-congruent, it suffices to find a curve $C$ with infinitely many rational points and a point on $C$ which does not correspond to isogenous curves. Since
the intersection of $C$ with $X_0(l)$ is either $X_0(l)$ or a finite set of points,
so only finitely many points on $C$ correspond to isogenous curves.

For each $j \neq 0,1728,\infty$, there exists a unique elliptic curve $E_a$
of the form $E: y^2=x^3+ax+a$ such that the $j$-invariant $j(E_a)=j$ and so
for $j \neq 0,1728,\infty$ we take the representative $E_a$ in the class
of $\mathbb{C}$-isomorphic elliptic curves containing $E_a$.

We start with the model $S_{a,1}=X_{E_a}(8)$ we got in the previous section and
now we consider $a$ being a variable. Then the irreducible part of $S_{a,1}$
with $a \neq 0,-\frac{27}{4}$ gives an open subscheme of $Z_{8,1}$, which we call $Z_8$. In [KS], it is shown that the $Z_{8,1}$ is a rational surface and we will verify this and give explicit birational map between $S_{a,1}$ and $\mathbb{P}^2$.
\begin{proposition}
The explicit birational map $\mathbb{A}^2_{p,q} \to S_{a,1}$ is given by
$(p,q) \mapsto (a,t,a_0,a_1,a_2)$ where
$$a=\frac{-8(q + 1)(q^2 + 2)^2h_1(p,q)^3}{
(q-1)(q^4 + 3q^2 - 2p)(q^6 + 3q^4 + q^2 - p^2 - 1)^2h_2(p,q)},$$
$$t=-\frac{(q^2 + 2)h_3(p,q)h_4(p,q)}{3(q - 1)(q^4 + 3q^2 - 2p)(q^6 + 3q^4 + q^2 - p^2 - 1)h_5(p,q)}$$
\begin{align*}
h_1(p,q)&=q^6 - \frac{3}{2}q^5 + 3q^4 + \frac{1}{2}q^3p - \frac{9}{2}q^3 - \frac{3}{2}q^2p + \frac{1}{2}q^2 + 3qp - q - \frac{1}{2}p^2 + \frac{1}{2}\\
h_2(p,q)&=q^6 - 3q^5 + 3q^4 + q^3p - 9q^3 - 3q^2p + 6qp - 2q + 2\\
h_3(p,q)&=q^6 - \frac{3}{2}q^5 + 3q^4 + \frac{1}{2}q^3p - \frac{9}{2}q^3 - \frac{3}{2}q^2p + \frac{1}{2}q^2 +3qp - q - \frac{1}{2}p^2 + \frac{1}{2}\\
h_4(p,q)&=q^{10} - 2q^9 +
10q^8 + 2q^7p - 8q^7 - 8q^6p + 26q^6 +12q^5p - 6q^5 \\
&+ q^4p^2 - 32q^4p + 16q^4 - 6q^3p^2+6q^3p - 4q^3 + 11q^2p^2 \\
&- 16q^2p + q^2 + 4qp - 4q +2p^2 + 2\\
h_5(p,q)&=q^6 - 3q^5 + 3q^4 + q^3p - 9q^3 - 3q^2p + 6qp - 2q + 2\\
\end{align*}
and we do not give images $a_0,a_1,a_2$ here due to massive expressions
(we will show how to obtain them in the proof).
\end{proposition}
\begin{proof}
We start with $S_{a,1}$ (setting $b=a$ in Theorem 1.1) and $S_{a,1}$ is birational to
the surface defined by the following equations:
\begin{align*}
f'_1&=-a + 2a_0 + a^2_1+2a^2_2,\\
f'_2&=-2aa_1 - a + 2a_0a_1-2ta_2,\\
f'_3&=-2aa_1 + a^2_0+aa^2_2-t^2.\\
\end{align*}
by $(a,t,a_0,a_1,a_2) \mapsto (a,\frac{1}{3a_2},\frac{a_0}{a_2},\frac{a_1}{a_2},\frac{-t}{a_2})$.
We use $f'_1,f'_2$ to write $a,t$ in terms of $a_0,a_1,a_2$
$$a=2a_0+a^2_1+2a^2_2,t=\frac{-2aa_1-a+2a_0a_1}{2a_2}.$$
Then we have
\begin{align*}
f'&=-a^2_0a^2_1 - 2a^2_0a_1 + a^2_0a^2_2 - a^2_0 - 2a_0a^4_1 - 3a_0a^3_1 -
    4a_0a^2_1a^2_2 - a_0a^2_1 - 10a_0a_1a^2_2\\
&+ 2a_0a^4_2 - 2a_0a^2_2 - a^6_1 -a^5_1 - 4a^4_1a^2_2 - \frac{1}{4}a^4_1 - 6a^3_1a^2_2 - 3a^2_1a^4_2 - a^2_1a^2_2\\
&-8a_1a^4_2 + 2a^6_2 - a^4_2.
\end{align*}
Sending $a_0$ to $\frac{Aa_0+B}{a^3_2}$ and $a_1$ to $\frac{a_1}{a_2}$ where
\begin{align*}
A&=-a^2_1 - 2a_1 + a^2_2 - 1,\\
B&=-2a^4_1 - 3a^3_1 - 4a^2_1a^2_2 - a^2_1 - 10a_1a^2_2 + 2a^4_2 - 2a^2_2
\end{align*}
we conclude the surface is birational to the vanishing of
$$h'=a^2_0-(a^6_1a^2_2 + 3a^5_1a_2 + 3a^4_1a^2_2 + \frac{9}{4}a^4_1 + 9a^3_1a_2 + a^2_1a^2_2 + 9a^2_1 +2a_1a_2 - a^2_2 + 1).$$
This is a genus zero curve defined over $\mathbb{Q}(a_1)$ with
a rational point given by
$$a_0=\frac{3a_1-\frac{3}{2}a^2_1+\frac{1}{2}a^3_1}{a_1-1},
a_2=-\frac{1}{a_1-1}.$$
Hence using standard parametrization of the genus zero curve with a rational
point, together with the intermediate steps we worked above we obtain the
required parametrization.
\end{proof}
The proposition allows one to classify all elliptic curves $E$
such that the curve $X_E(8)$ has non-trivial rational points.
\begin{corollary}
There exist infinitely many pairs of non-isogenous elliptic curves
which are directly $8$-congruent.
\end{corollary}
\begin{proof}
This follows immediately from the Remark above and Proposition 6.1.
\end{proof}

We now consider a similar construction of the modular surface corresponding
to $X^r_E(8)$ which helps to find numerical examples for $r=3,5,7$. We start with the model $X^r_E(8)$, set $b=a$ and view $a$ as a variable. We call
the resulting variety $S_{a,r}$.
\begin{proposition}
There exists infinitely many pairs of non-isogenous elliptic curves
which are $8$-congruent with power $5$.
\end{proposition}
\begin{proof}
Consider the genus 0 curve on $S_{a,5}$ parameterized by $\mathbb{A}^1_p$, defined by $a_2=0$,
\begin{align*}
a&=\frac{27}{8}\frac{(p^2-12p+12)^2}{(p^2-12)^2},
t=-\frac{1}{2}\frac{p^2 - 12p + 12}{p^2-12},\\
a_0&=\frac{-243}{32}\frac{(p^2-12p+12)^3(p^2 - 4p + 12)}{(p^2-12)^4},\\
a_1&=\frac{81}{8}\frac{(p^2 - 12p + 12)^2(p^2 - 4p + 12)}{(p^2-12)^3}.
\end{align*}
Setting $p=2$ we obtain a pairs of elliptic curves
\begin{align*}
E:& y^2 = x^3 + 54x + 216 & \quad 20736p1\\
F:& y^2 = x^3 - 522x + 18936 &  \quad 103680bv1\\
\end{align*}
which are non-isogenous and $8$-congruent with power $5$.
\end{proof}
If $E$ is $5$-isogenous to $F$ then $F$ is $8$-congruent to $E$
with power $5$ and therefore we have a copy of $X_0(5)$
on $S_{a,5}$ which corresponds to pairs of of $5$-isogenous curves.
Let $E_r:y^2=x^3+a_rx+b_r$ be the families of elliptic curves parameterized by $X_0(5)$, where
\begin{align*}
a_r&=-27r^4 + 324r^3 - 378r^2 - 324r -27,\\
b_r&=54r^6 - 972r^5 + 4050r^4 + 4050r^2 + 972r + 54\\
\end{align*}
The $F_r$, which is $5$-isogenous to $E_r$, is
$$y^2+(1-r)xy-ry=x^3-rx^2-5t(r^2+2r-1)x-r(r^4+10r^3-5r^2+15r-1)$$
and has $j$-invariant
$\frac{(r^4 + 228r^3 + 494r^2 - 228r + 1)^3}{r(r^2 - 11r - 1)^5}$.
By considering the $j$-map $X^4_E(8) \to X(1)$ we obtain the value of $t$
which corresponds to $F_r$. Then the point on $X^5_E(8)$
corresponding to $F_r$ is
\begin{align*}
t&=r^2+1,~a_0=-1944r(r^3 - 11r^2 + 7r + 1)(r^3 - 7r^2 - 11r - 1),\\
a_1&=324r(r^2 - 12r - 1)(r^2 + 1),~a_2=108r(r^2-6r-1).
\end{align*}
To find a copy of $X_0(5)$ on $S_{a,5}$ we replace both $a_r,b_r$ by
$a^3_r/b^2_r$ and rescale the coordinates $t,a_0,a_1,a_2$.
Similarly we can find points corresponding to $3$-isogenous curves on $S_{a,3}$ and points corresponding to $7$-isogenous curves on
$S_{a,7}$.
\begin{proposition}
There are infinitely many pairs of non-isogenous elliptic curves which are 8-congruent with power 3.
\end{proposition}
\begin{proof} We search for a genus $0$ curve on $S_{a,3}$. Consider the genus 0 curve on $S_{a,3}$ parameterised by $\mathbb{A}^1_r$
\begin{align*}
a&=-\frac{135}{4}\frac{(r^2 - 2r - \frac{15}{8})(r^2 + \frac{1}{8})}{(r^2 - r + \frac{11}{8})(r^2 + r + \frac{3}{8})(r^2 + 2r - \frac{1}{8})},\\
t&=\frac{1}{2}\frac{(r^2 - 2r - \frac{15}{8})(r^2 + \frac{1}{8})}{(r^2 - r + \frac{11}{8})(r^2 + r + \frac{3}{8})},\\
a_0&=-135\frac{(r^2 - 2r + \frac{21}{8})(r^2 - \frac{1}{2}r - \frac{3}{8})
(r^2+\frac{1}{8})^2(r^2+\frac{1}{2}r+\frac{5}{8})(r^2+\frac{6}{5}r+\frac{17}{40})}
{(r^2-r+\frac{11}{8})^2(r^2+r+\frac{3}{8})^2(r^2+2r-\frac{1}{8})^3},\\
a_1&=0,a_2=6\frac{(r^2 - \frac{1}{2}r - \frac{3}{8})^2(r+\frac{1}{8})}{(r^2 - r + \frac{11}{8})(r^2+r+\frac{3}{8})(r^2+2r-\frac{1}{8})}.
\end{align*}
Then the curves corresponding to $r=0$ are non-isogenous and the result follows.
\end{proof}
The families of curves parameterized by $X_0(3)$ are $E_r: y^2=x^3+a_rx+b_r$ where
$a_r=18r-27,b_r=9r^2-54r+54$.
So the curve $F_r$ which is $3$-isogenous to $E_r$ corresponds to
$$t=1-r,a_0=36r^2 - 126r + 108,a_1=15r-18,a_2=3r-6$$
on $X^3_{E_r}(8)$.
\begin{proposition}
There are infinitely many pairs of non-isogenous elliptic curves which are 8-congruent
with power 7.
\end{proposition}
\begin{proof} We start with the model $S_{a,7}$ and we take the section $t=0$.
Then we obtain a curve $C$ which has 2 irreducible components, one of which is not reduced. We take the reduced one, say $C_1$, which is a genus 1 curve and it has a rational point
$$p:a=-9,t=0,a_0=3,a_1=1,a_2=0$$
and so $C_1$ is isomorphic to
$$C': y^2 = x^3 + x^2 - 538x + 4628$$
which has rank $1$. Finally, we search a point on $C_1$ given by
$$a=-\frac{135}{32},t=0, a_0=\frac{75}{32}, a_1=\frac{5}{4},a_2=\frac{-1}{3}$$
and this point gives a pair of non-isogenous curves
$$E_1:y^2 = x^3 - 1080x - 17280,E_2:y^2 = x^3 + 7931250x - 8519850000.$$
\end{proof}
The families of curves parameterized by $X_0(7)$ are $E_r:y^2=x^3+a_rx+b_r$ where
\begin{align*}
a_r&=-27r^8 + 324r^7 - 1134r^6 + 1512r^5 -945r^4 + 378r^2 - 108r - 27,\\
b_r&=54r^{12} - 972r^{11} + 6318r^{10} -19116r^9 + 30780r^8 - 26244r^7\\
 &+ 14742r^6 - 11988r^5 + 9396r^4 - 2484r^3 - 810r^2 + 324r + 54,\\
\end{align*}
and so the curve $F_r$ which is $7$-isogenous to $E_r$ corresponds to
\begin{align*}
t&=\frac{r^6 - 7r^5 - 14r^4 + 53r^3 - 34r^2 + r + 1}{r^2 - r + 1},\\
a_0&=12(-r^8 + 15r^7 - 72r^6 + 125r^5 - 113r^4 + 48r^3 + 5r^2 - 7r - 1),\\
a_1&=\frac{2r^6 - 26r^5 + 80r^4 - 50r^3 - 20r^2 + 14r + 2}{r^2 - r + 1},
a_2=\frac{2}{3},
\end{align*}
on $X^7_{E_r}(8)$.
\begin{flushleft}
\bf{Remark}
\end{flushleft}
The surface $Z_{8,7}$ is a surface of general type, and one might expect to take more effort to find rational points on $S_{a,7}$.

We now give some examples with small conductors.
\begin{flushleft}
\bf{Example}
\end{flushleft}
By searching rational points on $X^5_E(8)$ we obtain a curve $F$ which is non-isogenous to $E$ and
$8$-congruent to $E$ with power $5$, where
\begin{align*}
E: &y^2 = x^3 + x^2 - 17x - 33               &\quad   96a2\\
F: &y^2 = x^3 - 8x^2 - 333056x + 59636736   &\quad   1056d2
\end{align*}
We give the traces of Frobenius of the curves
at first several places
\begin{center}
    \begin{tabular}{ | l | l | l | l| l |l|l|l|l|l|l|l|}
    \hline
  Prime & 2  & 3& 5& 7 & 11 &13 & 17& 19& 23& 29 &31 \\ \hline

  Traces of Frobenius(E) &0&1&2&-4&4&-2&-6&-4&0&2&4 \\ \hline
  Traces of Frobenius(F) &0&1&2&4&-1&-2&2&4&0&-6&4\\ \hline
\end{tabular}
\end{center}
and they are congruent mod $8$ except $p=11$. Further, at $p=3$, both curves
have split multiplicative reduction and we have $v_3(\Delta_E)=1,v_3(\Delta_F)=5$ which agrees with Proposition 2 in [KO].

\begin{flushleft}
\bf{Example}
\end{flushleft}
By searching rational points on $X^3_E(8)$ we obtain a curve $F$ which is non-isogenous to $E$ and
$8$-congruent to $E$ with power $3$, where
\begin{align*}
E:  &y^2 + xy + y = x^3 - x^2 - 2x            &\quad 99a1\\
F:  &y^2 = x^3 - 975159243x + 11681563877190  &\quad 1683b1
\end{align*}
We give the traces of Frobenius of the curves at first several places
\begin{center}
    \begin{tabular}{ | l | l | l | l| l |l|l|l|l|l|l|l|l|}
    \hline
  Prime & 2 & 3&5&7&11&13&17&19&23&29&31\\ \hline

  Traces of Frobenius(E)& 7&0&4&6&7&6&2&2&4&2&4   \\ \hline
  Traces of Frobenius(F)& 7&0&4&6&-7&6&1&2&4&2&4  \\ \hline
 \end{tabular}
\end{center}
and they are congruent mod $8$ except $p=11,17$. Further at $p=11$ both curves
have split multiplicative reduction and we have $v_{11}(\Delta_E)=1$,
$v_{11}(\Delta_F)=3$ which agree with Proposition 2 in [KO].

We made some effort to minimise and reduce the equation $X_E(8)$ and find some examples of triples of elliptic curves which are directly 8-congruent to each other.
\begin{example}
Elliptic curves $129a1,645e1$ and $23349a1$ are directly 8-congruent to each other.
\end{example}
\begin{example}
Elliptic curves $561a1,235059g1$ and $171105h1$ are directly 8-congruent to each other.
\end{example}

\begin{flushleft}
\bf{Appendix}
\end{flushleft}
Let $E:y^2=x^3+ax+b$ be an elliptic curve. Let $c_4=-\frac{a}{27},c_6=-\frac{b}{54}$.  The families of elliptic curves parameterised by $X^3_E(4)$ are
$$E^{\Delta_E}_t:y^2=x^3-27\Delta^2_Ea_E(t)x-54\Delta^3_Eb_E(t)$$
where
\begin{align*}
a_E(t)&=c_4t^8 + 8c_6t^7 + 28c^2_4t^6 + 56c_4c_6t^5 + (-42c^3_4 +
112c^2_6)t^4\\
& +56c^2_4c_6t^3 + (252c^4_4 - 224c_4c^2_6)t^2 + (264c^3_4c_6 - 256c^3_6)t +
    (81c^5_4 - 80c^2_4c^2_6),\\
b_E(t)&=c_6t^{12} + 12c^2_4t^{11} + 66c_4c_6t^{10} + (44c^3_4 + 176c^2_6)t^9 +
    495c^2_4c_6t^8 \\
&+ 792c^4_4t^7 + 924c^3_4c_6t^6+ (-2376c^5_4 +3168c^2_4c^2_6)t^5 + (-5841c^4_4c_6 + 6336c_4c^3_6)t^4\\
& + (-1188c^6_4 -4224c^3_4c^2_6 + 5632c^4_6)t^3 + (-4158c^5_4c_6 + 4224c^2_4c^3_6)t^2 \\
&+(-2916c^7_4 + 4464c^4_4c^2_6 - 1536c_4c^4_6)t+ (-1215c^6_4c_6 +2240c^3_4c^3_6 - 1024c^5_6).
\end{align*}

{\bf University of Cambridge, DPMMS, Centre for Mathematical Sciences, Wilberforce
Road, Cambridge CB3 0WB, UK}\\
e-mail address: zc231@cam.ac.uk
\end{document}